%% file: main.tex
\documentclass[12pt,twoside]{article}
\usepackage{etex}
\usepackage[dvipsnames]{xcolor} 
\usepackage{color}
\usepackage{tcolorbox}
\usepackage{booktabs}
\usepackage{amsthm}
\usepackage{listings} 
\usepackage{amsfonts,amssymb,amsxtra,url,float} 
\allowdisplaybreaks[4] 
\usepackage[colorlinks,
  linkcolor=magenta, %
  anchorcolor=Periwinkle,
  citecolor=red,
  urlcolor=blue
  ]{hyperref} 
\usepackage{enumitem}
\usepackage{geometry} 
\usepackage{rotating} 
\usepackage{lscape} 
\usepackage{multirow}
\usepackage{graphicx} 
\usepackage{subfigure} 
\usepackage{tikz}
\usepackage{pgfplots}
\usepackage{tikz-3dplot}
\usetikzlibrary{patterns}
\usetikzlibrary{3d,calc}
\usetikzlibrary{decorations.pathreplacing,decorations.markings}
 \tikzset{
  on each segment/.style={
    decorate,
    decoration={
      show path construction,
      moveto code={},
      lineto code={
        \path [#1]
        (\tikzinputsegmentfirst) -- (\tikzinputsegmentlast);
      },
      curveto code={
        \path [#1] (\tikzinputsegmentfirst)
        .. controls
        (\tikzinputsegmentsupporta) and (\tikzinputsegmentsupportb)
        ..
        (\tikzinputsegmentlast);
      },
      closepath code={
        \path [#1]
        (\tikzinputsegmentfirst) -- (\tikzinputsegmentlast);
      },
    },
  },
  mid arrow/.style={postaction={decorate,decoration={
        markings,
        mark=at position 0.6 with {\arrow[#1]{stealth}} 
      }}},
}
\usetikzlibrary{arrows}
\usetikzlibrary{trees}

\usetikzlibrary{matrix}
\usetikzlibrary{patterns}
\usetikzlibrary{shadings} 
\usepackage{fancyhdr} 
\def\headertitle{Phantom subcategories on blow-ups of Hirzebruch surfaces}
\def\fstpage{1} 
\def\page{$\begin{matrix} {\color{white}0} \\ \thepage \end{matrix}$} 

\usepackage[all]{xy} 
\usepackage{dsfont} 
\usepackage{cite}
\usepackage{mathrsfs} 
\numberwithin{figure}{section}
\usepackage{marginnote} 
\usepackage{graphicx} 
\usepackage{multicol} 

\usepackage{enumitem}
\setenumerate[1]{itemsep=0pt,partopsep=0pt,parsep=\parskip,topsep=3pt}
\setitemize[1]{itemsep=0pt,partopsep=0pt,parsep=\parskip,topsep=3pt}
\setdescription{itemsep=0pt,partopsep=0pt,parsep=\parskip,topsep=3pt}
\setlist[itemize]{leftmargin=35pt}
\setlist[enumerate]{leftmargin=35pt}
\usepackage{changepage} 

\usepackage{contour} 
\usepackage{xcolor}
\contourlength{0.03em}
\contournumber{auto}

\newtheorem{theorem}{Theorem}[section]
\newtheorem{lemma}[theorem]{Lemma}
\newtheorem{corollary}[theorem]{Corollary}
\newtheorem{main theorem}[theorem]{Main Theorem}
\newtheorem{proposition}[theorem]{Proposition}
\newtheorem{definition}[theorem]{Definition}

\newtheorem{remark}[theorem]{Remark}

\newtheorem{conjecture}[theorem]{Conjecture}

\numberwithin{equation}{section}

\def\orcid{
\begin{tikzpicture}[baseline=-1mm]
\filldraw[Green!35] (0,0) circle (5pt);
\filldraw[white] (0,0) node{\tiny\textbf{iD}};
\end{tikzpicture}
}

\usepackage{extarrows}

\def\orcid{
\begin{tikzpicture}[baseline=-1mm]
\filldraw[Green!35] (0,0) circle (5pt);
\filldraw[white] (0,0) node{\tiny\textbf{iD}};
\end{tikzpicture}
}

\title{\bf
Phantom subcategories on blow-ups of Hirzebruch surfaces
}

\vspace{5mm}

\author{
Zeyuan He$^{\ref{Author1}}$,
Yu-Zhe Liu$^{\ref{Author1}, \href{https://orcid.org/0009-0005-1110-386X}{\orcid}\ref{orcid2},\ref{CorrespondingAuthor}}$,
Mingzhi Sheng$^{\ref{Author3}}$,
Panyue Zhou$^{\ref{Author4}}$
}
\date{ }
\begin{document}






\maketitle

\begin{enumerate}[label=\textbf{\color{red}\arabic*}] \footnotesize
  \item
    \begin{center}
      School of Mathematics and Statistics,
      Guizhou University, Guiyang 550025, Guizhou, China;
    \end{center}
    \label{Author1}


  \item
    \begin{center}
      School of Science, Hunan City University, Yiyang, 413000, Hunan, China;
    \end{center}
    \label{Author3}

  \item
    \begin{center}
      School of Mathematics and Statistics, Changsha University of Science and Technology,

      410114 Changsha, Hunan, China;
    \end{center}
    \label{Author4}
  \item[]
    \begin{center}
      E-mail:
      \url{liuyz@gzu.edu.cn} / \url{yzliu3@163.com} (Y. Z. Liu);

      \url{panyuezhou@163.com} (P. Zhou)
    \end{center}
\end{enumerate}
\vspace{1mm}
\begin{enumerate}[label=\textbf{\color{red}$\ddag$}]
  \item \footnotesize
    \begin{center}
      Corresponding author
    \end{center} \label{CorrespondingAuthor}
\end{enumerate}
\vspace{1mm}
\begin{enumerate}[leftmargin=6.5cm] \footnotesize
  \item[\orcid]
      ORCID: \href{https://orcid.org/0009-0005-1110-386X}{0009-0005-1110-386X}
      \label{orcid2} 
\end{enumerate}

\hspace{2mm}



\begin{adjustwidth}{1cm}{1cm} 
\noindent \footnotesize 
\textbf{Abstract.}
We prove that the blow-up of the Hirzebruch surface $\mathbb F_n$ at $n+6$ general points admits a phantom subcategory for every $n\geqslant 3$.

\noindent\textbf{Key words.}
Hirzebruch surface, phantom category, exceptional collection, blow-up

\noindent\textbf{2010 Mathematics Subject Classification.}
14F08 
14J26 
18G80 

\end{adjustwidth}

%
%
%




\setcounter{tocdepth}{1}
\setcounter{secnumdepth}{3}


\def\gldim{\mathrm{gl.dim}}
\def\perm{\mathrm{perm}}
\def\forb{\mathrm{forb}}
\def\rmv{\mathsf{v}}
\def\rmp{\mathsf{w}}


\section{Introduction}

\def\kk{\Bbbk}
\def\CC{\mathbb{C}}
\def\FF{\mathbb{F}}
\def\PP{\mathbb{P}}
\def\ZZ{\mathbb{Z}}
\def\>={\geqslant}
\def\=<{\leqslant}

\def\HHomolo{\mathrm{HH}}
\def\Homolo{\mathrm{H}}

\def\calA{\mathcal{A}}
\def\calB{\mathcal{B}}
\def\calC{\mathcal{C}}
\def\calD{\mathcal{D}}
\def\calE{\mathcal{E}}
\def\calF{\mathcal{F}}
\def\calG{\mathcal{G}}
\def\calH{\mathcal{H}}
\def\calI{\mathcal{I}}
\def\calJ{\mathcal{J}}
\def\calK{\mathcal{K}}
\def\calL{\mathcal{L}}
\def\calM{\mathcal{M}}
\def\calN{\mathcal{N}}
\def\calO{\mathcal{O}}
\def\calP{\mathcal{P}}
\def\calQ{\mathcal{Q}}
\def\calR{\mathcal{R}}
\def\calS{\mathcal{S}}
\def\calT{\mathcal{T}}
\def\calU{\mathcal{U}}
\def\calV{\mathcal{V}}
\def\calW{\mathcal{W}}
\def\calX{\mathcal{X}}
\def\calY{\mathcal{Y}}
\def\calZ{\mathcal{Z}}

\def\bfP{\mathbf{P}}
\newcommand{\bfp}{\boldsymbol p}

\def\Bl{\mathrm{Bl}}
\def\Open{\mathrm{Open}}
\def\op{\mathrm{op}}
\def\Conf{\mathrm{Conf}}
\def\phac{\mathrm{ph}_{\mathrm{ac}}}
\def\ph{\mathrm{ph}}

\def\Hom{\mathrm{Hom}}
\def\End{\mathrm{End}}
\def\Ext{\mathrm{Ext}}
\def\Picagroup{\mathscr{P}\!ic}

\def\Tot{\mathrm{Tot}}
\def\divisor{\mathrm{div}}
\def\Div{\mathrm{Div}}
\def\Aff{\mathbf{A}}

\def\Ring{\mathsf{Ring}}
\def\Modcat{\mathsf{Mod}}
\def\Coh{\mathsf{Coh}}
\def\Dcat{\mathsf{D}}

\newcommand{\defines}[1]{{\color{blue!75}\it #1}}

Let $X$ be a smooth projective variety over $\CC$.
A nonzero admissible subcategory $\calP\subseteq\Dcat^b(X)$ is called a phantom if its Grothendieck group and Hochschild homology vanish, and Gorchinskiy and Orlov constructed phantom categories in \cite{GO2013}.
B\"{o}hning, Graf von Bothmer, Katzarkov, and Sosna obtained further examples from exceptional collections on Barlow surfaces \cite{BGKS2015}.
In \cite{Krah2024}, Krah constructed a phantom on the blow-up of $\PP^2$ at ten general points.
His construction uses a non-full exceptional collection of line bundles of maximal length.
Kemboi et al.\ constructed phantoms on the blow-up of $\PP^2$ at eleven general points and on the blow-up of $\FF_2$ at nine general points \cite[Theorems 1.1 and 1.3]{KKLetc2025}.
The existence part of their Conjecture 4.11 predicts a phantom on the blow-up of $\FF_n$ at $6+\max\{3,n\}$ general points \cite[Conjecture 4.11]{KKLetc2025}.

In this paper, we prove this existence statement for every $n \>= 3$. The result is stated in Theorem \ref{thm:main-intro}.
We first study the blow-up of $\FF_3$ at nine points.
We construct thirteen line bundles using an involution of the Picard lattice (Lemma~\ref{lem:isometry}).
We check the required vanishing by interpolation (Proposition \ref{prop:computation}).
Kuznetsov's pseudoheight criterion shows that the resulting exceptional collection is not full \cite[Corollary 6.2]{Kuz2015}.
Its orthogonal is a phantom by Theorem \ref{thm:base}.

For $n>3$, we use elementary transformations to obtain a blow-up of $\FF_3$ (Lemma~\ref{lem:contraction}).
We then transport the phantom by Orlov's blow-up formula \cite{O1992}.
Section \ref{sec:tools} collects the tools used in the proof.
Section \ref{sec:base} treats $\FF_3$, and Section~\ref{sec:propagation} treats the remaining cases.
The {proof of the existence part of} \cite[Conjecture 4.11]{KKLetc2025} in this paper used \texttt{DeepSeek V 4.1 Flash} and \texttt{ChatGPT 6.0}.

\section{Preliminaries}\label{sec:preliminaries}

\textsl{This section recalls the definitions of Hirzebruch surfaces, blow-ups, phantom subcategories, and states the conjecture of Kemboi et al.}

\subsection{Hirzebruch surfaces}

Let $\CC$ be the complex field in this paper, and let $\CC^{\times} = \CC\backslash\{0\}$.
A \defines{{\rm(}complex{\rm)} projective line} $\PP^1$ is defined as
\[ \PP^1 = \{[x_0:x_1]: (x_0,x_1)\in \CC^2\backslash\{(0,0)\}\}, \]
whose element $[x_0:x_1] = \{(cx_0, cx_1): c\in \CC^{\times}\}$ is a subset of $\CC^2\backslash\{(0,0)\}$.
Thus, $[x_0:x_1]=[x_0':x_1']$ if and only if there is a $c\in \CC^{\times}$ such that $x_0'=cx_0$ and $x_1'=cx_1$.
One can check that $\PP^1 = \{[t:1]: t\in\CC\} \cup \{[1:0]\} \cong \CC\cup\{\infty\}$,
cf. \cite[Chapter I, Section 2]{Hart1977}.

Let $\Open(\PP^1)$ be a category whose objects are open sets of $\PP^1$,
and whose morphisms are given by the the inclusion relation ``$\subseteq$'' of sets.
Let $\Ring$ be the category of rings.
Consider the following functor, which is called a \defines{structure sheaf} in algebraic geometry,
\[ \calO:=\calO_{\PP^1} : \Open(\PP^1)^{\op} \to \Ring, \quad U \mapsto \calO(U), \]
which sends each open set $U$ of $\PP^1$ to the ring $\calO(U) = \{f:U \to \CC : f \text{ is a regular function}\}$,
see \cite[Chapter II, Sections 1--2]{Hart1977}.
Write
\begin{align*}
 & U_0 = \{[1:t]: t\in \CC\}, \\
 & U_1 = \{[s:1]: s\in \CC\}.
\end{align*}
Then, on the intersection $\CC^{\times} = U_0\cap U_1$, we have $s=\frac{1}{t}$.
Given an integer $n\in\ZZ$, a \defines{Serre twisting sheaf} $\calO(n)$ of $\calO$,
see \cite[Chapter II, Section 5]{Hart1977}, is the functor
\[ \calO(n):= \calO_{\PP^1}(n): \Open(\PP^1)^{\op} \to \mathsf{Ab}, \]
which sends each open set $U$ of $\PP^1$ to the left $\calO_{\PP^1}(U)$-module $\calO(n)(U) = \{(f_0, f_1)\in \calO(U\cap U_0) \times \calO(U\cap U_1): f_0(t) = t^n f_1(\frac{1}{t}) = s^{-n}f_1(s)$ on the subset $U\cap \CC^{\times}\}$,
where:
\begin{itemize}
  \item the pair $(f_0,f_1)$ consists of local representatives of a section in two trivializations; the representatives need not define a single $\CC$-valued function on $U$;
  \item the left $\calO_{\PP^1}$-action is induced by
    \[ \calO(U) \times \calO(n)(U) \to \calO(n)(U), (g,(f_0,f_1))\mapsto (g|_{U\cap U_0}f_0,g|_{U\cap U_1}f_1) \]
    (thus, since this holds for every open set $U$, we say that $\calO(n)$ is a left $\calO_{\PP^1}$-module).
\end{itemize}
In particular, we have $\calO(0)=\calO$.

For each $p\in\PP^1$ define
\[ \calO_p = \mathop{\underrightarrow{\lim}}\limits_{p\in U} \calO_{\PP^1}(U)
\text{ and } \calO(n)_p = \mathop{\underrightarrow{\lim}}\limits_{p\in U} \calO_{\PP^1}(n)(U). \]
Here $\mathfrak m_p=\{[f]_p\in\calO_p:f(p)=0\}$ is the unique maximal ideal of the local ring $\calO_p$.
Then the fiber $\calO(p):=\calO_p/\mathfrak m_p\calO_p$ corresponding to $\calO_p$ is isomorphic to $\CC$,
and the fiber $\calO(n)(p):=\calO(n)_p/\mathfrak m_p\calO(n)_p$ corresponding to $\calO(n)_p$ is isomorphic to $\CC$.
One can check that the fiber of $\calO\oplus \calO(n)$ at the point $p\in\PP^1$, regarded as $\CC^{\oplus 2}$, is a two-dimensional vector space, whose projectivization, say $F_{n,p}$, is isomorphic to a projective line $\PP^1$.

\begin{definition}[{\cite[Chapter V, Section 2]{Hart1977}}]
Take a base space $\PP^1 = \CC\cup\{\infty\}$. A \defines{Hirzebruch surface} is defined as a fiber bundle
\[ \FF_n := \bfP_{\PP^1}(\calO\oplus\calO(n)) := \bigsqcup_{p\in\PP^1} F_{n,p}
 \cong \{ (p,\ell) : p \in \PP^1,\ \ell\in F_{n,p} \}. \]
\end{definition}

\noindent
In particular, for $n=0$, we have
\[ \FF_0 \cong \PP^1 \times \PP^1. \]

\subsection{Blow-ups}

Let $X$ be a smooth surface over $\CC$, and let $p\in X$ be a closed point.

\begin{definition} \rm \
\begin{enumerate}[label={\rm(\arabic*)}]
  \item \cite[Chapter I, Section 1]{Hart1977} An \defines{irreducible curve} on $X$ is a closed subset of $X$ of dimension one such that it cannot be written as the union of two proper closed subsets.
  \item \cite[Chapter I, Section 1]{Beauville1996} For two distinct irreducible curves $c_1$ and $c_2$ on $X$ meeting transversely at every point of $c_1\cap c_2$, their \defines{intersection number} $c_1\cdot c_2$ is defined as $|c_1\cap c_2|$.
  \item \cite[Chapter II, Section 6]{Hart1977} A \defines{divisor} on $X$ is a formal $\ZZ$-linear combination of irreducible curves, i.e., an element of the free abelian group generated by the irreducible curves of $X$. The group of divisors is denoted by $\Div(X)$.
\end{enumerate}
\end{definition}

\begin{definition}[{\!\!\cite[Chapter II, Section 6]{Hart1977}}] \rm
The \defines{Picard group} $\Picagroup(X)$ is the group of isomorphism classes of line bundles on $X$, with tensor product as its group operation. Since $X$ is smooth, it is naturally identified with the group of divisors modulo linear equivalence.
Here, $D\sim D'$ means that $D-D'=\divisor(f)$ for some nonzero rational function $f$ on $X$.
\end{definition}

\begin{definition}[{\!\!\cite[Chapter V, Section 3]{Hart1977}}] \rm
The \defines{blow-up} of $X$ at $p$ is a smooth surface $\Bl_p X$ together with a morphism
\[\pi_{X,p}: \Bl_p X \to X\]
such that:
\begin{itemize}
  \item $\pi_{X,p}$ is an isomorphism over $X\backslash\{p\}$,
  \item the fiber $\pi_{X,p}^{-1}(p)$ is a curve $E\cong\PP^1$, called the \defines{exceptional curve}.
\end{itemize}
\end{definition}

Locally, choose local coordinates $(x,y)$ on a neighborhood $U$ centered at $p$; then we have
\[\pi_{X,p}^{-1}(U) = \{((x,y),[u:v])\in U\times\PP^1 \mid xv=yu\},\]
and $E=\pi_{X,p}^{-1}(0,0)=\{0\}\times\PP^1$. One has $E^2=-1$.
More generally, for distinct points $p_1,\ldots,p_r\in X$, we write
\[ \pi_{X, p_1,\ldots,p_r}:\Bl_{p_1,\dots,p_r}X \to X\]
for the blow-up at points $p_1,\ldots,p_r$, with exceptional curves $E_1,\ldots,E_r$. They satisfy
\[E_i^2=-1,\text{ and } E_i\cdot E_j=0 \  (i\neq j).\]
For a divisor $D$ on $X$, we denote by $\pi_{X, p_1,\ldots,p_r}^*(D)$ its pullback to $\Bl_{p_1,\dots,p_r}X$, which is characterized by
\[ \pi_{X, p_1,\ldots,p_r}^*(D)\cdot\pi_{X, p_1,\ldots,p_r}^*(D') = D\cdot D'
 \quad\text{and}\quad \pi_{X, p_1,\ldots,p_r}^*(D)\cdot E_i=0 ~ (1\=< i\=< r).\]

\subsection{Phantom subcategories}

Recall that a \defines{sheaf of $\calO_X$-modules}, or simply an \defines{$\calO_X$-module}, is a sheaf $\calF$ on $X$ such that for every open subset $U\subseteq X$, the group $\calF(U)$ is an $\calO_X(U)$-module, and the restriction maps are compatible with the module structure.
Here, an $\calO_X$-module is \defines{of finite type} if it is locally generated by finitely many sections.
See \cite[Chapter II, Section 5]{Hart1977}.

\begin{definition}[{See \cite[Section 2, no. 12, Definition 2]{Serre1955}, or \cite[Tag 01BV]{StacksProject}}] \rm
An $\calO_X$-module $\calF$ is \defines{coherent} if it is of finite type and the kernel of every morphism $\calO_X^{\oplus m}|_U\to\calF|_U$ is of finite type for every open $U\subseteq X$ and every $m\>= 0$.
We write $\Coh(X)$ for the Abelian category of coherent sheaves on $X$, and $\Dcat^b(X)$ for its bounded derived category.
\end{definition}

For objects $E,F \in \Dcat^b(X)$, we write
\[ \Hom(E,F[r]) = \Ext^r(E,F) \]
for the group of morphisms of degree $r$ ($\in\ZZ$).

\begin{definition}[{\!\!\cite[Section 1.1]{KuzPol2016}}] \rm
An object $E\in\Dcat^b(X)$ is \defines{exceptional} if
\[ \Hom(E,E[r]) =
\begin{cases}
  \CC, & r=0,\\
    0, & r\ne 0.
\end{cases} \]
An ordered collection $(E_0,\ldots,E_s)$ of exceptional objects is an \defines{exceptional collection} if
$\Hom(E_j,E_i[r])$ $= 0$ holds for all $i<j$ and $r\in\ZZ$.
We call that it is \defines{full} if it generates $\Dcat^b(X)$.
\end{definition}

A full triangulated subcategory $\calA\subseteq\Dcat^b(X)$ is \defines{admissible} if
its inclusion has both a left and a right adjoint, see \cite[Definition 1.1]{GO2013}.
For an exceptional collection $\calE=(E_0,\ldots,E_s)$, the subcategory it generates is admissible,
and its \defines{right orthogonal}
\[ \calE^\perp = \{T\in\Dcat^b(X) : \Hom(E_i,T[r])=0 \text{ for all } i,r\} \]
appears in a \defines{semiorthogonal decomposition} (see {\cite[Section 2.1]{KKLetc2025}}) $\Dcat^b(X)=\langle \calE^\perp,E_0,\ldots,E_s\rangle$,
where our convention for $\langle\calA,\mathcal{B}\rangle$ is that $\Hom(B,A[r])=0$ for $A\in\calA$, $B\in\calB$, and $r\in\ZZ$.

\begin{definition}[{\!\!\cite[Definition 1.8]{GO2013}}] \rm
A \defines{phantom subcategory} of $\Dcat^b(X)$ is a nonzero admissible subcategory $\calP$ satisfying
\[ K_0(\calP)=0,\quad \HHomolo_q(\calP)=0\quad(\forall q\in\ZZ). \]
Here, $K_0(\calP)$ is the Grothendieck group and $\HHomolo_q(\calP)$ is the Hochschild homology of $\calP$.
\end{definition}

The Grothendieck group $K_0(\calP)$ and the Hochschild homology $\HHomolo_q(\calP)$ both are additive under semiorthogonal decompositions.
An exceptional object contributes $\ZZ$ to the Grothendieck group and $\CC$, concentrated in degree zero, to Hochschild homology.
In particular, the two vanishing conditions in the definition can be checked once an exceptional collection of the appropriate length is known to be non-full.

\subsection{The conjecture of Kemboi et al.}

For a smooth projective surface $S$ and an integer $r\>= 1$, let
\[ \Conf_r(S) = \{(p_1,\dots,p_r)\in S^r : p_i\neq p_j \text{ for } i\neq j\}. \]
This is called the \defines{configuration space} of ordered $r$-tuples of distinct points of $S$.
A property is said to \defines{hold for general points} if it holds on a nonempty Zariski open subset of $\Conf_r(S)$.
The following {is the existence part of a conjecture in} \cite{KKLetc2025}.

\begin{conjecture}[{\cite[Conjecture 4.11]{KKLetc2025}}]
\label{conj:kemboi}
For every integer \(n\>= 0\), let $d(n)=6+\max\{3,n\}$.
Then for general points $\bfp=(p_1,\dots,p_{d(n)})\in\Conf_{d(n)}(\FF_n)$, the derived category
$\Dcat^b(\Bl_{\bfp}\FF_n)$ contains a phantom subcategory.
\end{conjecture}

The cases $n \in \{0,1\}$ follow from Krah's construction on the blow-up of $\PP^2$ at ten general points \cite{Krah2024}, and the case $n=2$ was proved in \cite{KKLetc2025}. The main result of this paper is the following theorem, which proves the {existence part of the} conjecture for all $n\>= 3$.

\begin{theorem} \label{thm:main-intro}
For every integer $n\>= 3$, there exists a nonempty Zariski open subset $U_n\subseteq \Conf_{n+6}(\FF_n)$
such that for every $\bfp\in U_n$, the derived category $\Dcat^b(\Bl_{\bfp}\FF_n)$ contains a phantom subcategory.
\end{theorem}

Together with the known cases $n=0,1,2$, Theorem \ref{thm:main-intro} establishes {the existence part of} Conjecture \ref{conj:kemboi} for every nonnegative integer $n$.

\section{Some categorical and numerical tools}\label{sec:tools}

From now on all surfaces are smooth and projective over $\CC$, and all derived subcategories are understood to be closed under direct summands.
We use the standard enhancements when taking Hochschild homology.
For $n\>= 0$, let
\begin{itemize}
  \item $C_n$ be the section of self-intersection $-n$ on $\FF_n$ (for $n=0$, fix either ruling);
  \item $F$ be a fiber class.
\end{itemize}
On $X=\Bl_{p_1,\ldots,p_r}\FF_n$, with points away from $C_n$,
we use the same letters for the pullback classes and write $E_i$ for the exceptional curves. The intersection numbers and canonical class are
\begin{gather}
 C_n^2=-n,\quad C_n\cdot F=1,\quad F^2=0,\quad
 E_i\cdot E_j=-\delta_{ij},\quad C_n\cdot E_i=F\cdot E_i=0,
 \label{eq:intersection-general}\\
 K_X := -2C_n-(n+2)F+\sum_{i=1}^r E_i.
 \label{eq:canonical-general}
\end{gather}
The class $F$ is numerically effective.

For a coherent sheaf $\calF$ on $X$, let $\Homolo^i(X,\calF)$ denote its $i$-th cohomology group, and write $h^i(X,\calF):=\dim_{\CC}\Homolo^i(X,\calF)$. In particular, $\Homolo^0(X,\calF)$ is the space of global sections of $\calF$.
Then $B\cdot F<0$ implies $\Homolo^0(X,\calO_X(B))=0$. Since $X$ is rational, Riemann--Roch and
Serre duality show that
\begin{equation}\label{eq:RR-general}
 \chi(X,\calO_X(B))=1+\frac{1}{2} B\cdot(B-K_X), \text{ and }
 h^2(X,\calO_X(B))=h^0(X,\calO_X(K_X-B)).
\end{equation}
See \cite[Chapter V, Section 2]{Hart1977} for ruled surfaces.

We shall also use Orlov's blow-up formula \cite{O1992}.
If $\pi:X\to Y$ is the blow-up of a smooth projective surface at $r$ distinct points,
with exceptional curves $\Gamma_i$, then $L\pi^*$ is fully faithful and
\begin{equation}\label{eq:orlov}
\Dcat^b(X)=
  \bigl\langle
    \calO_{\Gamma_1}(-1),\ldots,\calO_{\Gamma_r}(-1), L\pi^*\Dcat^b(Y)
  \bigr\rangle.
\end{equation}
All displayed components are admissible.

\begin{lemma}\label{lem:additive-invariants}
If $X$ is the blow-up of $\FF_n$ at $r$ distinct points, then
\[ K_0(X)\cong\ZZ^{r+4},\quad  \HHomolo_q(X)=0\ (q\ne0), \quad \dim_\CC\HHomolo_0(X)=r+4. \]
\end{lemma}

\begin{proof}
The projective bundle formula gives $K_0(\FF_n)\cong\ZZ^4$.
The blow-up formula adds one copy of $\ZZ$ for each point.
The Hodge numbers of $X$ vanish off the diagonal,
while $h^{0,0}(X)=h^{2,2}(X)=1$ and $h^{1,1}(X)=r+2$.
Hochschild--Kostant--Rosenberg gives the asserted Hochschild homology.
\end{proof}

For objects $E,F\in\Dcat^b(X)$, put
\[ e(E,F) = \inf\{q\in\ZZ: \Ext^q(E,F) \ne 0\}, \]
with value $+\infty$ if all these groups vanish. For an exceptional collection of line bundles $\calE=(L_0,\ldots,L_s)$ on a smooth projective surface, its anticanonical pseudoheight is
\begin{equation}\label{eq:anti-pseu}
 \phac(\calE)=
 \min_{0\=< i_0<\cdots<i_p\=< s}
 \left\{\sum_{a=0}^{p-1}e(L_{i_a},L_{i_{a+1}})
 +e(L_{i_p},L_{i_0}\otimes\calO_X(-K_X))-p\right\}.
\end{equation}
The minimum includes $p=0$. By~\cite[Definition 4.9]{Kuz2015}, the pseudoheight is $\ph(\calE)=\phac(\calE)+2$.

\begin{lemma}\label{lem:nonfull-criterion}
Let $\calE=(L_0,\ldots,L_s)$ be an exceptional collection of line
bundles on a smooth projective surface. If $\Hom(L_i,L_j)=0$ for all
$i<j$, then $\ph(\calE)\>=2$. In particular, $\calE$ is not full.
\end{lemma}

\begin{proof}
The first $p$ terms in \eqref{eq:anti-pseu} are at least one. The last term is nonnegative because negative Ext groups
between line bundles vanish.
Thus, $\phac(\calE)\>=0$ and $\ph(\calE)\>=2$.
The last assertion follows from \cite[Corollary 6.2]{Kuz2015}.
\end{proof}

\section{A phantom on the blow-up of \texorpdfstring{$\FF_3$}{F3} at nine points}
\label{sec:base}

Let $X=\Bl_{p_1,\ldots,p_9}\FF_3$, where the points lie away from the negative section and on distinct fibers. Write $C=C_3$ and $K=K_X$.
Thus, we have
\begin{equation}\label{eq:base-canonical}
 K=-2C-5F+\sum_{i=1}^9E_i,\qquad K^2=-1,\qquad K\cdot C=1.
\end{equation}
Define divisor classes
\begin{align}
 D_i & = 3K+C-E_i = -5C-15F+3\sum_{j=1}^9 E_j-E_i  \quad (1\=< i\=<9),\label{eq:Di}\\
   G & = 5K+C-F   = -9C-26F+5\sum_{j=1}^9 E_j. \label{eq:G}
\end{align}
We consider the ordered collection
\begin{equation}\label{eq:collection}
 \calE = (\calO_X,\calO_X(D_1),\ldots,\calO_X(D_9),
            \calO_X(G),\calO_X(C+3G),\calO_X(C+4G)).
\end{equation}
Write $A_0,\ldots,A_{12}$ for the corresponding divisor classes and
put $L_i=\calO_X(A_i)$.

\begin{theorem}\label{thm:base}
There is a nonempty Zariski open subset $U_3\subset\Conf_9(\FF_3)$ such that, for every configuration in $U_3$,
the collection~\eqref{eq:collection} is exceptional and not full.
Its right orthogonal $\calP_3=\calE^\perp$ is a phantom subcategory of $\Dcat^b(X)$.
\end{theorem}

The proof occupies the remainder of this section.
We first construct the divisor classes by an isometry and verify numerical exceptionality.
We then prove the necessary vanishing by interpolation.
Finally, we apply Lemma \ref{lem:nonfull-criterion} and compute the additive invariants of the orthogonal.

\subsection{An isometry fixing the negative section}

\begin{lemma}\label{lem:isometry}
The map
\begin{equation}\label{eq:isometry}
 J(B)=-B-\bigl(3B\cdot K+B\cdot C\bigr)K -\bigl(B\cdot K+B\cdot C\bigr)C
\end{equation}
is an integral isometry of $\Picagroup(X)$. It satisfies $J^2=1$, $J(K)=K$, $J(C)=C$, $J(E_i)=D_i$, and $J(F)=G$.
\end{lemma}

\begin{proof}
The intersection matrix on the rational span $W$ of $K$ and $C$ is
\[\begin{pmatrix}
  -1 & 1\\  1 & -3
\end{pmatrix}.\]
Its determinant is two. Let $J$ act as the identity on $W$ and as minus the identity on $W^\perp$. Then $J$ is an involutive isometry of $\Picagroup(X)\otimes\mathbb Q$.

Write $J(B)=-B+\alpha K+\beta C$. Taking the intersection product with $K$ and $C$, we obtain
\[ -\alpha+\beta=2B\cdot K, \text{ and }  \alpha-3\beta=2B\cdot C. \]
Solving these equations, we obtain
\[ \alpha=-3B\cdot K-B\cdot C, \text{ and } \beta=-B\cdot K-B\cdot C. \]
This proves \eqref{eq:isometry}. Both coefficients are integers for $B\in\Picagroup(X)$,
so $J$ preserves the integral lattice. Since $J^2=1$, its restriction is an integral isometry.
Finally, direct substitution shows that $J(E_i)=D_i$ and $J(F)=G$.
\end{proof}

The list $(A_0,\ldots,A_{12})$ is the image under $J$ of
\begin{equation}\label{eq:initial-list}
 (0,E_1,\ldots,E_9,F,C+3F,C+4F).
\end{equation}
We use $J$ only as a lattice isometry; no derived autoequivalence is asserted.
In particular, cohomology vanishing must be checked separately.

\begin{lemma}\label{lem:numerical}
  For $0\=< i<j\=<12$, one has $\chi(X,\calO_X(A_i-A_j))=0$.
\end{lemma}

\begin{proof}
For $B=aC+bF+\sum_i c_iE_i$, formula~\eqref{eq:RR-general} becomes
\begin{equation}\label{eq:RR-base}
 \chi(X,\calO_X(B)) = 1+\frac{-3a^2+2ab-\sum_i c_i^2-a+2b+\sum_i c_i}{2}.
\end{equation}
Up to the indices on the exceptional classes, the differences of an earlier term and a later term in \eqref{eq:initial-list} are
\begin{align*}
 &-E_i,   &&-F,       && -C-3F,    && -C-4F, && E_i-E_j,\\
 & E_i-F, && E_i-C-3F,&& E_i-C-4F, && -C-2F. &&
\end{align*}
Each has Euler characteristic zero by \eqref{eq:RR-base}. The map $J$ preserves the intersection form and $K$, so it preserves
the Euler characteristic in~\eqref{eq:RR-general}.
\end{proof}

\subsection{The interpolation conditions}

For distinct indices $i,j$, the exact sequence
\[ 0 \to \calO_X \to \calO_X(E_j) \to \calO_{E_j}(-1) \to 0 \]
shows that the only effective divisor in $|E_j|$ is $E_j$ itself.
It cannot contain $E_i$. Therefore,
\begin{equation}\label{eq:excep-diff}
 \Homolo^0(X,\calO_X(E_j-E_i))=0\quad(i\neq j).
\end{equation}
Since $D_i-D_j=E_j-E_i$, this treats $\binom{9}{2}=36$ of the $\binom{13}{2}=78$ differences $A_i-A_j$ with $i<j$.

The remaining differences are listed in Table \ref{tab:classes}.
They have the form $B=aC+bF-\sum_{j=1}^9m_jE_j$.
The notation $(u;v^8)_i$ means $m_i=u$ and $m_j=v$ for $j\neq i$,
and $(v^9)$ means that all nine multiplicities are $v$.

\begin{table}[htbp]
\centering\small
\caption{The remaining differences $A_i-A_j$, $i<j$.}
\label{tab:classes}
\begin{tabular}{@{}clrrcrr@{}}
\toprule
Type & Divisor & $a$ & $b$ & Multiplicities & $N$ & Pairs\\
\midrule
I & $D_i-G$ & 4 & 11 & $(3;2^8)_i$ & 30 & 9\\
II & $-D_i$ & 5 & 15 & $(2;3^8)_i$ & 51 & 9\\
III & $-G$ & 9 & 26 & $(5^9)$ & 135 & 2\\
IV & $-C-2G$ & 17 & 52 & $(10^9)$ & 495 & 1\\
V & $D_i-C-3G$ & 21 & 63 & $(13;12^8)_i$ & 715 & 9\\
VI & $-C-3G$ & 26 & 78 & $(15^9)$ & 1080 & 2\\
VII & $D_i-C-4G$ & 30 & 89 & $(18;17^8)_i$ & 1395 & 9\\
VIII & $-C-4G$ & 35 & 104 & $(20^9)$ & 1890 & 1\\
\bottomrule
\end{tabular}
\end{table}

Type III occurs for $(i,j)=(0,10),(11,12)$ and type VI for $(i,j)=(0,11),(10,12)$.
The other pairs are read directly from \eqref{eq:collection}.
Thus the table accounts for all remaining $42$ pairs and gives $4\cdot9+4=40$ distinct divisors with labelled
points. All nine choices of the distinguished index in types I, II, V, and VII are included in the computation below.

Let $\rho:X\to\FF_3$ be the blow-up morphism. For nonnegative $m_i$,
the projection formula gives
\begin{equation}\label{eq:fat-points}
 \Homolo^0(X,\calO_X(B))\cong
 \Homolo^0\bigg(
             \FF_3,\calO(aC+bF)\otimes
             \bigcap_{i=1}^9\mathcal I_{p_i}^{m_i}
          \bigg).
\end{equation}
Indeed, $\rho_*\calO_X(-\sum_i m_iE_i)=\bigcap_i\mathcal I_{p_i}^{m_i}$.

The complement of $C$ in $\FF_3$ is $\Tot\calO_{\PP^1}(3)$.
Choose an affine coordinate $t$ on the base and a fiber coordinate $z$.
On the other base chart, $t'=t^{-1}$ and $z'=t^{-3}z$.
For $a\>=0$, the sections of $\calO(aC+bF)$ have a basis
\begin{equation}\label{eq:monomials}
 t^rz^j,\quad 0\=< j\=< a,\quad 0\=< r\=< b-3j,
\end{equation}
where a range with negative upper bound is empty.
To see this, the degree in $z$ is at most $a$ by regularity along the section at infinity.
In a suitable local frame on the other base chart,
a section $f(t,z)$ is expressed as $t^{-b}f(t,t^3z')$. Its regularity gives $r+3j\=< b$.
Equivalently, for the ruling $f:\FF_3\to\PP^1$,
\[ f_*\calO(aC+bF)=\bigoplus_{j=0}^a\calO_{\PP^1}(b-3j). \]

At $p_i=(t_i,z_i)$, the condition of vanishing to order at least $m_i$ is the vanishing of the coefficients of $u^\alpha v^\beta$ in $f(t_i+u,z_i+v)$ for $\alpha+\beta<m_i$.
In the basis \eqref{eq:monomials}, the resulting matrix is
\begin{equation}\label{eq:matrix}
 M_{(i,\alpha,\beta),(j,r)} = \binom r\alpha\binom j\beta t_i^{r-\alpha}z_i^{j-\beta},
 \quad \alpha+\beta<m_i.
\end{equation}
The entry is defined to be zero if $r<\alpha$ or $j<\beta$.
These are Hasse derivatives, so the matrix is defined over the integers without division by factorials.
Its kernel is the space in \eqref{eq:fat-points}. In every row of Table~\ref{tab:classes},
\begin{equation}\label{eq:dimension}
 N = \sum_{j=0}^a\max\{b-3j+1,0\} = \sum_{i=1}^9\frac{m_i(m_i+1)}2.
\end{equation}
In particular, all forty matrices are square.

\subsection{Exact computation and a nonempty open set} \label{subsec:computation}

We use the same integer points for all forty matrices.
Their coordinates are given in Table \ref{tab:points}.
Rows of $M$ are ordered first by $i=1,\ldots,9$, then by $\alpha=0,\ldots,m_i-1$,
and finally by $\beta=0,\ldots,m_i-1-\alpha$.
Columns are ordered first by $j$ and then by $r$.
These conventions fix the signs of the determinants.

\begin{table}[htbp]
\centering\small
\caption{The integer configuration used in the computation.}
\vspace{3mm}
\label{tab:points}
\setlength{\tabcolsep}{3pt}
\begin{tabular}{@{}crrrrrrrrr@{}}
\toprule
$i$&1&2&3&4&5&6&7&8&9\\
\midrule
$t_i$&1&2&3&4&5&6&7&8&9\\
$z_i$&38495&64320&17708&5484&47004&37992&33632&20617&58817\\
\bottomrule
\end{tabular}
\end{table}

\begin{proposition}\label{prop:computation}
At the configuration in Table~\ref{tab:points}, each of the forty matrices \eqref{eq:matrix} associated with Table~\ref{tab:classes} is invertible over $\FF_{65521}$.
\end{proposition}
\begin{proof}[Proof by exact computation]
The determinants modulo $65521$ are listed in Tables \ref{tab:distinguished} and~\ref{tab:uniform}.
Every entry is nonzero. We specify the arithmetic procedure so that the finite calculation can be reproduced from the displayed data.

\begin{table}[htbp]
\centering\small
\caption{Determinants modulo $65521$ for types with a distinguished point $i$.}
\vspace{3mm}
\label{tab:distinguished}
\begin{tabular}{@{}crrrr@{}}
\toprule
$i$& I ($N=30$)& II ($N=51$)& V ($N=715$)& VII ($N=1395$)\\
\midrule
1 & 47624 & 36904 & 56830 & 29251\\
2 & 53768 & 15232 & 61922 & 29161\\
3 & 21304 & 19578 & 19598 & 16299\\
4 & 16174 & 53795 & 36911 & 12753\\
5 & 51872 & 14462 & 1242 & 38707\\
6 & 35359 & 62458 & 5317 & 6325\\
7 & 25790 & 59274 & 32643 & 13524\\
8 & 30971 & 6837 & 16190 & 30001\\
9 & 59508 & 56302 & 26936 & 36175\\
\bottomrule
\end{tabular}
\end{table}

\begin{table}[htbp]
\centering\small
\caption{Determinants modulo $65521$ for the uniform multiplicities.}
\vspace{3mm}
\label{tab:uniform}
\begin{tabular}{@{}crr@{}}
\toprule
Type & $N$ & Determinant\\
\midrule
III & 135 & 15681\\
IV & 495 & 48983\\
VI & 1080 & 51907\\
VIII & 1890 & 11089\\
\bottomrule
\end{tabular}
\end{table}

Set $p=65521$, reduce the entries of~\eqref{eq:matrix} modulo $p$,
and apply Gaussian elimination over $\FF_p$.
More explicitly, start with $d=1$.
At column $c$, choose the first row $s\>= c$ with $M_{s,c}\ne 0$.
If there is no such row, the determinant is zero.
Otherwise interchange rows $s,c$ if necessary, changing the sign of $d$.
Set $a=M_{c,c}$ and replace $d$ by $da$. For each row $i>c$,
put $\lambda=M_{i,c}a^{-1}$ and perform
\[ M_{i,j}\longleftarrow M_{i,j}-\lambda M_{c,j} \quad(j>c),\quad M_{i,c}\longleftarrow0.\]
After the last column, $d$ is the determinant.
Inversion can be computed as $a^{-1}=a^{p-2}$ in $\FF_p$.

The displayed coordinates, matrix formula, row and column order,
and elimination rule specify the finite calculation completely.
The code implementing this calculation is listed in Appendix \ref{app:interpolation-code}.
\end{proof}

\begin{remark}\label{rem:65521} \rm
The modulus $65521$ used to verify invertibility of the matrices has precedents in computational mathematics.
First, $65521$ is the largest prime smaller than the Fermat prime number $2^{2^4}+1=2^{16}+1=65537$,
so every residue modulo $65521$ fits in an unsigned 16-bit integer.
One application is the Adler-32 checksum used in the zlib format.
Its two sums are computed modulo $65521$, see \cite{RFC1950}.
Since PNG image data are stored in zlib streams, the same checksum also occurs in PNG \cite{RFC2083}.
In Adler-32, the prime modulus helps avoid a class of two-byte errors that would otherwise go undetected \cite[Section 8.2]{RFC1950}.
The zlib technical notes also explain that larger sums and a prime modulus reduce some checksum collisions. 
By contrast, the Fletcher-32 algorithm uses the composite modulus $65535=2^{16}-1$,
whereas Adler-32 uses the prime modulus $65521$. 
The field $\FF_{65521}$ has also been used in published computations of matrix rank defects \cite{Faugere2010}.
For instance, a LinBox tutorial uses $\FF_{65521}$ for an exact determinant computation. 
The rationale for using $65521$ in this paper is the exact finite-field computation.
The residues fit in unsigned 16-bit integers, although the implementation uses wider integers for intermediate arithmetic.
A determinant that is nonzero modulo $65521$ is nonzero as an integer.
Since its entries are polynomial functions of the point coordinates, its determinant is a nonzero polynomial.
Therefore, the corresponding interpolation condition holds on a nonempty Zariski open set of configurations.
\end{remark}

\begin{lemma}\label{lem:open-set}
There is a nonempty Zariski open subset $U_3\subset\Conf_9(\FF_3)$ such that
$\Homolo^0(X,\calO_X(A_i-A_j))=0$ for every $i<j$ and every configuration in $U_3$.
\end{lemma}

\begin{proof}
Regarding the point coordinates as variables,
the determinants in Proposition \ref{prop:computation} are polynomials over $\ZZ$.
At the integer configuration of Table \ref{tab:points},
each determinant is nonzero modulo $65521$, hence is a nonzero integer.
In particular, the same configuration over $\CC$ makes all forty determinants nonzero.

Let $\Delta$ be their product, and let
\[ T_9=\bigl\{(t_i,z_i)_{i=1}^9\in\Aff^{18}: t_i\neq t_j\text{ for }i\neq j\bigr\}. \]
This is an open subset of $\Conf_9(\FF_3)$. The subset
\begin{equation}\label{eq:U3}
 U_3=T_9\cap\{\Delta\neq0\}
\end{equation}
is nonempty and Zariski open. On this set, the forty kernels in \eqref{eq:fat-points} vanish. The other thirty-six differences are covered by \eqref{eq:excep-diff}.
\end{proof}

\subsection{The proof of Theorem \ref{thm:base}}

\begin{proposition}\label{prop:exceptional}
For every configuration in $U_3$, the collection $\calE$ is exceptional.
\end{proposition}
\begin{proof}
Let $B=A_i-A_j$ with $i<j$. By Lemmas \ref{lem:numerical} and \ref{lem:open-set},
we have $\chi(X,\calO_X(B))=0$ and $h^0(X,\calO_X(B))=0$.
The coefficient of $C$ in $B$ is nonnegative for every difference under consideration.
Thus, $(K-B)\cdot F=-2-B\cdot F<0$.
Since $F$ is numerically effective, Serre duality gives $h^2(X,\calO_X(B))=0$.
Riemann--Roch then gives $h^1(X,\calO_X(B))=0$. Consequently,
\[ \Ext^q(L_j,L_i)=\Homolo^q(X,\calO_X(A_i-A_j))=0 \quad (i<j,\ q\in\ZZ).\]
Therefore, every line bundle on the rational surface $X$ is exceptional,
since $\Homolo^q(X,\calO_X)=0$ for $q>0$ and $\Homolo^0(X,\calO_X)=\CC$.
\end{proof}

\begin{proposition}\label{prop:nonfull}
For every configuration in $U_3$, one has $\ph(\calE)\>=2$. In particular, $\calE$ is not full.
\end{proposition}

\begin{proof}
The coefficients of $C$ in $A_0,\ldots,A_{12}$, in order, are
\[ 0,\underbrace{-5,\ldots,-5}_{9\text{ entries}},-9,-26,-35. \]
For $i<j$, unless both terms are among $D_1,\ldots,D_9$,
it follows that $(A_j-A_i)\cdot F<0$. Hence $\Homolo^0(X,\calO_X(A_j-A_i))=0$.
When both terms are among the $D_i$, the same conclusion follows from \eqref{eq:excep-diff}.
We have proved $\Hom(L_i,L_j)=0$ for every $i<j$.
Therefore, Lemma \ref{lem:nonfull-criterion} implies $\ph(\calE)\>= 2$, and $\calE$ is not full.
\end{proof}

Now we can prove Theorem \ref{thm:base}.

\begin{proof}[Proof of Theorem \ref{thm:base}]
By Propositions~\ref{prop:exceptional} and~\ref{prop:nonfull}, the
collection $\calE$ is exceptional and not full. Its right orthogonal
$\calP_3$ is therefore nonzero and admissible, with a semiorthogonal
decomposition
\begin{equation}\label{eq:base-sod}
 \Dcat^b(X)=\langle\calP_3,L_0,\ldots,L_{12}\rangle.
\end{equation}
By Lemma~\ref{lem:additive-invariants} and additivity of $K_0$,
\[
 \ZZ^{13}\cong K_0(X)\cong K_0(\calP_3)\oplus\ZZ^{13}.
\]
The group $K_0(\calP_3)$ is a direct summand of a finitely generated free
abelian group and has rank zero. Hence $K_0(\calP_3)=0$.

Again by Lemma~\ref{lem:additive-invariants}, $\HHomolo_q(X)=0$ for $q\neq0$
and $\dim_\CC\HHomolo_0(X)=13$. Each of the thirteen exceptional objects
in~\eqref{eq:base-sod} contributes one copy of $\CC$ in degree zero.
Additivity of Hochschild homology gives $\HHomolo_q(\calP_3)=0$ for every
$q\in\ZZ$. Thus $\calP_3$ is a phantom subcategory.
\end{proof}

\section{Elementary transformations and the blow-up formula}
\label{sec:propagation}

\textsl{We now deduce Theorem \ref{thm:main-intro} from Theorem \ref{thm:base}.
The geometric step is to contract suitable strict transforms of fibers.
We also give an explicit map of configurations to ensure that the resulting nine points belong to the open subset $U_3$.}

\subsection{Elementary transformations}

\begin{lemma}\label{lem:elementary}
Let $n\>=2$ and $p\in\FF_n\setminus C_n$.
Blow up $p$, and contract the strict transform $\Gamma$ of the fiber through $p$.
The resulting surface is $\FF_{n-1}$, and the image of $C_n$ is its negative section.
\end{lemma}

\begin{proof}
On the blow-up, $\Gamma=F-E$ satisfies $\Gamma^2=-1$, so it can be contracted.
The ruling descends to a $\PP^1$-bundle over $\PP^1$.
The strict transform of $C_n$ meets $\Gamma$ transversely once,
and its image is a section of self-intersection $-n+1<0$.
The classification of ruled surfaces identifies the resulting surface
with $\FF_{n-1}$; see~\cite[Chapter V, Section 2]{Hart1977}.

On the affine chart in $\Tot\calO_{\PP^1}(n)$, if $p=(a,b)$, the
transformation has the expression
\[ (t,z)\mapsto \left(t,\frac{z-b}{t-a}\right).\]
The base point is resolved by blowing up $p$.
The rest of the fiber $t=a$ is mapped to the point on the negative section over $a$.
\end{proof}

\begin{lemma}\label{lem:contraction}
Let $n\>=3$, put $k=n-3$, and let $X=\Bl_{p_1,\ldots,p_{k+9}}\FF_n$.
Suppose that the points lie away from $C_n$ and on distinct fibers.
There is a morphism
\begin{equation}\label{eq:contraction}
 \pi:X \to  Y=\Bl_{q_1,\ldots,q_9}\FF_3
\end{equation}
contracting the $k$ disjoint curves $\Gamma_i=F-E_i$, $1\=< i\=< k$.
The points $q_1,\ldots,q_9$ are the images of $p_{k+1},\ldots,p_{k+9}$ under the corresponding elementary transformations.
\end{lemma}

\begin{proof}
First blow up $p_1,\ldots,p_k$ and contract the strict transforms of their fibers.
These are disjoint $(-1)$-curves. Repeated application of Lemma \ref{lem:elementary} reduces the index from $n$ to $n-k=3$.
Each elementary transformation is an isomorphism away from its fiber.
The last nine points lie away from these fibers, so their blow-ups commute with the elementary transformations.
This shows \eqref{eq:contraction}.
\end{proof}

Write $C_Y,F_Y$ for the negative section and fiber classes on $Y$.
On $X$, the pullback classes satisfy
\begin{equation}\label{eq:pullback-classes}
 \pi^*C_Y=C_n+\sum_{i=1}^k(F-E_i),\quad
 \pi^*F_Y=F,\quad
 K_X=\pi^*K_Y+\sum_{i=1}^k\Gamma_i.
\end{equation}
In particular, $\pi$ is the blow-up of $k$ distinct points on the strict transform of $C_3$ in $Y$. The last nine blow-up points are away from that section. The points blown up by $\pi$ need not be general in $Y$.

\subsection{General configurations}

For any $s$, let $T_s$ denote the affine open set of ordered coordinates $(t_i,z_i)_{i=1}^s$ with pairwise distinct $t_i$.
When working on $\FF_n$, it is viewed in the chart of $\Tot\calO_{\PP^1}(n)$ over $\Aff^1$.
Thus, $T_s$ is an open subset of $\Conf_s(\FF_n)$.

\begin{lemma}\label{lem:configuration-map}
For every $n\>=3$, the elementary transformations in Lemma \ref{lem:contraction} define,
after a choice of coordinates, a surjective morphism
\[ \Psi_n:T_{n+6} \to  T_9. \]
Consequently, the condition that $(q_1,\ldots,q_9)$ belong to $U_3$ holds on a nonempty Zariski open subset of $\Conf_{n+6}(\FF_n)$.
\end{lemma}

\begin{proof}
Put $k=n-3$. If $k>0$, define
\begin{align}
 h(t)&=\prod_{i=1}^k(t-t_i),\label{eq:polynomial-h}\\
 g(t)&=\sum_{i=1}^k z_i
       \prod_{\substack{1\=<\ell\=< k\\\ell\neq i}}
       \frac{t-t_\ell}{t_i-t_\ell}.\label{eq:polynomial-g}
\end{align}
Then $\deg g<k$ and $g(t_i)=z_i$ for $1\=< i\=< k$. For $k=0$, set $h=1$ and $g=0$.

Since $\deg g\=< n$, translation by $-g$ is an automorphism of $\Tot\calO_{\PP^1}(n)$ extending to $\FF_n$ and fixing $C_n$.
It sends the first $k$ points to the zero section.
Performing the elementary transformation at each of these points then gives
\begin{equation}\label{eq:birational-map}
 (t,z)\dashrightarrow(t,w),\quad w=\frac{z-g(t)}{h(t)}.
\end{equation}
The division by $h$ changes the degree of the line bundle from $n$ to $n-k=3$.
More explicitly, on the other base chart,
put $h_\infty(t')=t^{-k}h(t)$ and $g_\infty(t')=t^{-n}g(t)$.
Then
\[ w'=\frac{z'-g_\infty(t')}{h_\infty(t')} =t^{-3}w.\]
Here, we have $h_\infty(0)=1$, so no further elementary transformation occurs over infinity.

The remaining nine points are therefore mapped to
\begin{equation}\label{eq:Psi}
 q_j=\left(t_{k+j},
 \frac{z_{k+j}-g(t_{k+j})}{h(t_{k+j})}\right),
 \quad 1\=< j\=<9.
\end{equation}
All denominators are nonzero on $T_{k+9}$. Thus~\eqref{eq:Psi} defines a morphism $\Psi_n:T_{k+9}\to T_9$.

To prove surjectivity, let $(s_j,w_j)_{j=1}^9\in T_9$. Choose distinct $t_1,\ldots,t_k$ avoiding all the $s_j$, and choose any $z_1,\ldots,z_k$. Form $g,h$ as above and set
\[ t_{k+j}=s_j,\quad z_{k+j}=g(s_j)+h(s_j)w_j \quad(1\=< j\=<9). \]
This configuration belongs to $T_{k+9}$ and maps to the prescribed one. It follows that
\begin{equation}\label{eq:Un}
 U_n=\Psi_n^{-1}(U_3)
\end{equation}
is nonempty and open in $T_{k+9}$, hence in $\Conf_{n+6}(\FF_n)$.
\end{proof}

\subsection{Phantoms under blow-up}

\begin{lemma}\label{lem:phantom-pullback}
Let $\pi:X\to Y$ be the blow-up of a smooth projective surface at finitely many distinct points. If $\calP\subset\Dcat^b(Y)$ is a phantom subcategory, then $L\pi^*(\calP)\subset\Dcat^b(X)$ is a phantom subcategory equivalent to $\calP$.
\end{lemma}

\begin{proof}
By \eqref{eq:orlov}, the functor $L\pi^*$ is fully faithful and its image is admissible.
The composite of the inclusion of $\calP$ with this functor is again an admissible embedding,
since both inclusions have left and right adjoints.
Hence $L\pi^*(\calP)$ is nonzero, admissible, and equivalent to $\calP$.
The equivalence preserves $K_0$ and Hochschild homology, proving the claim.
\end{proof}

\begin{proof}[Proof of Theorem~\ref{thm:main-intro}]
Let $U_n$ be as in~\eqref{eq:Un} and $X=\Bl_{\bfp}\FF_n$ with $\bfp\in U_n$.
Lemmas~\ref{lem:contraction} and \ref{lem:configuration-map} give
\[ \pi:X \to Y=\Bl_{\boldsymbol q}\FF_3, \quad \boldsymbol q\in U_3. \]
By Theorem \ref{thm:base}, the category $\Dcat^b(Y)$ contains the phantom
$\calP_3$. Lemma \ref{lem:phantom-pullback} gives the phantom $\calP_n=L\pi^*(\calP_3)$ in $\Dcat^b(X)$.
More precisely, with $k=n-3$, the semiorthogonal decompositions \eqref{eq:orlov} and \eqref{eq:base-sod} give
\begin{equation}\label{eq:general-sod}
 \Dcat^b(X)=\bigl\langle
 \calO_{\Gamma_1}(-1),\ldots,\calO_{\Gamma_k}(-1),
 \calP_n,L\pi^*L_0,\ldots,L\pi^*L_{12}
 \bigr\rangle.
\end{equation}
When $n=3$, $\pi$ is the identity and the first $k$ components are absent. This proves the theorem for every $n\>=3$.
\end{proof}

We finish by combining Theorem~\ref{thm:main-intro} with the previously
known cases.

\begin{corollary}\label{cor:all-n}
For every nonnegative integer $n$, set $d(n)=6+\max\{3,n\}$.
The blow-up of $\FF_n$ at $d(n)$ general points admits a phantom subcategory.
In particular, the required numbers of points are
\[ d(n)= \begin{cases}
   9, & 0\=< n\=<3,\\
 n+6, & n\>=4.
 \end{cases} \]
\end{corollary}

\begin{proof}
The assertion for $n\>=3$ is Theorem~\ref{thm:main-intro}.
For $n=2$, it is \cite[Theorem 1.3]{KKLetc2025}.
For $n=0,1$, it follows from Krah's construction on the blow-up of $\PP^2$ at ten general
points \cite{Krah2024}, as observed in the introduction to~\cite{KKLetc2025}.
Indeed, $\FF_1$ is the blow-up of $\PP^2$ at one point,
and the blow-up of $\FF_0$ at one point is isomorphic to the blow-up of $\PP^2$ at two distinct points. These descriptions identify the corresponding surfaces for general remaining points.
\end{proof}

\begin{remark} \rm
Corollary \ref{cor:all-n} establishes existence at the point count in \cite[Conjecture 4.11]{KKLetc2025}.
It does not assert that this is the smallest possible number of points.
The phantoms obtained for $n>3$ are equivalent to phantoms on the corresponding nine-point blow-ups of $\FF_3$.
\end{remark}

\paragraph{Authors' Contributions}

The order of authors is alphabetical, and all authors contributed equally to the conception, methodology, derivation, and writing of this paper.

\paragraph{Competing Interests}

The {authors} declare that they have no conflicts of interest as defined by the journal, nor any other interests that could be perceived as influencing the results presented in this paper.

\paragraph{Data availability}
{The point coordinates, determinant values, and code used in the computation are included in this paper and its appendix.}

\paragraph{Ethical Approval}

This article does not require ethical approval.

\paragraph{Funding}
Yu-Zhe Liu is supported by
the National Natural Science Foundation of China (Grant Nos. 12401042 and 12561008),
the Science and Technology Foundation of the Guizhou S~\&~T Department (Grant Nos. KJLYRC[2026]091, VZD[2026]001, ZD[2025]085 and ZK[2024]YiBan066),
and Scientific Research Foundation of Guizhou University (Grant No. [2023]16).

\appendix

\input{interpolation_matrix_appendix.tex}

%

\end{document}

%% file: interpolation_matrix_appendix.tex
\section{Appendix: code for the interpolation matrices}\label{app:interpolation-code}

The following Python 3 program constructs the forty matrices in \eqref{eq:matrix} using the point coordinates in Table \ref{tab:points}.
It computes their determinants over $\mathbb F_{65521}$ by Gaussian elimination. It requires NumPy. Running the program without arguments checks all eight types.
Passing type names as arguments checks only those types, for example, \texttt{python verify\_matrices.py I VII}.
Each output line contains the type, the distinguished point (or a dash), the matrix size, the determinant modulo $65521$, and the elapsed time.

\begingroup
\lstset{
  language=Python,
  basicstyle=\ttfamily\scriptsize,
  breaklines=true,
  columns=fullflexible,
  keepspaces=true,
  showstringspaces=false,
  frame=single
}
\begin{lstlisting}
import math
import sys
import time
import numpy as np

p = 65521
ts = [1, 2, 3, 4, 5, 6, 7, 8, 9]
zs = [38495, 64320, 17708, 5484, 47004, 37992, 33632, 20617, 58817]
types = [
    ('I', 4, 11, 2, 3), ('II', 5, 15, 3, 2),
    ('III', 9, 26, 5, None), ('IV', 17, 52, 10, None),
    ('V', 21, 63, 12, 13), ('VI', 26, 78, 15, None),
    ('VII', 30, 89, 17, 18), ('VIII', 35, 104, 20, None),
]

def matrix(a, b, mults):
    cols = [(j, r) for j in range(a + 1) for r in range(b - 3*j + 1)]
    rows = [(i, alpha, beta) for i, m in enumerate(mults)
            for alpha in range(m) for beta in range(m-alpha)]
    n = len(cols)
    assert len(rows) == n
    out = np.empty((n, n), dtype=np.int64)
    coeffs = {(r, alpha): math.comb(r, alpha) % p
              for _, r in cols for alpha in range(max(mults)) if alpha <= r}
    jcoeffs = {(j, beta): math.comb(j, beta) % p
               for j, _ in cols for beta in range(max(mults)) if beta <= j}
    for row, (i, alpha, beta) in enumerate(rows):
        t, z = ts[i], zs[i]
        out[row] = [coeffs.get((r, alpha), 0) * jcoeffs.get((j, beta), 0)
                    * pow(t, r-alpha, p) * pow(z, j-beta, p) % p
                    if r >= alpha and j >= beta else 0 for j, r in cols]
    return out

def determinant(a):
    n = a.shape[0]
    d = 1
    for c in range(n):
        pivot_rows = np.flatnonzero(a[c:, c])
        if not len(pivot_rows):
            return 0
        s = c + int(pivot_rows[0])
        if s != c:
            a[[c, s]] = a[[s, c]]
            d = -d
        pivot = int(a[c, c])
        d = d * pivot % p
        if c + 1 < n:
            factors = (a[c+1:, c] * pow(pivot, -1, p)) % p
            # Work in blocks to control temporary memory.
            for lo in range(c+1, n, 128):
                hi = min(n, lo+128)
                a[lo:hi, c+1:] -= factors[lo-c-1:hi-c-1, None] * a[c, c+1:]
                a[lo:hi, c+1:] %= p
            a[c+1:, c] = 0
    return d % p

names = set(sys.argv[1:])
for name, a, b, regular, special in types:
    if names and name not in names:
        continue
    for distinguished in range(9) if special is not None else [None]:
        ms = [regular] * 9
        if distinguished is not None:
            ms[distinguished] = special
        start = time.time()
        mat = matrix(a, b, ms)
        print(name, distinguished + 1 if distinguished is not None else '-',
              mat.shape[0], determinant(mat),
              f'{time.time()-start:.1f}s', flush=True)
\end{lstlisting}
\endgroup